\documentclass[12pt]{article}
\usepackage{amsmath,amssymb,amsthm,amscd,a4wide,upref}
\usepackage[enableskew]{youngtab}
\usepackage[colorlinks=true, linkcolor=blue, citecolor=magenta, urlcolor=blue, backref=page]{hyperref} 

\hypersetup{colorlinks,citecolor=blue,filecolor=black,linkcolor=blue,urlcolor=blue}
\usepackage{enumerate}
\usepackage{mathtools}
\usepackage{enumitem}

\usepackage{lmodern}     
\usepackage[T1]{fontenc}
\usepackage{indentfirst}

\begin{document}

\newcommand{\End}{{\rm{End}\ts}}
\newcommand{\Hom}{{\rm{Hom}}}
\newcommand{\Inv}{{\rm{Inv}}}
\newcommand{\Mat}{{\rm{Mat}}}
\newcommand{\ch}{{\rm{ch}\ts}}
\newcommand{\sh}{{\rm{sh}}}
\newcommand{\chara}{{\rm{char}\ts}}
\newcommand{\diag}{ {\rm diag}}
\newcommand{\non}{\nonumber}
\newcommand{\wt}{\widetilde}
\newcommand{\wh}{\widehat}
\newcommand{\ot}{\otimes}
\newcommand{\la}{\lambda}
\newcommand{\La}{\Lambda}
\newcommand{\De}{\Delta}
\newcommand{\al}{\alpha}
\newcommand{\be}{\beta}
\newcommand{\ga}{\gamma}
\newcommand{\Ga}{\Gamma}
\newcommand{\ep}{\epsilon}
\newcommand{\ka}{\kappa}
\newcommand{\vk}{\varkappa}
\newcommand{\si}{\sigma}
\newcommand{\vp}{\varphi}
\newcommand{\de}{\delta}
\newcommand{\ze}{\zeta}
\newcommand{\om}{\omega}
\newcommand{\ee}{\epsilon^{}}
\newcommand{\su}{s^{}}
\newcommand{\hra}{\hookrightarrow}
\newcommand{\ve}{\varepsilon}
\newcommand{\ts}{\,}
\newcommand{\vac}{\mathbf{1}}
\newcommand{\vacr}{|\tss 0\rangle}
\newcommand{\vacl}{\langle 0\tss |}
\newcommand{\di}{\partial}
\newcommand{\qin}{q^{-1}}
\newcommand{\tss}{\hspace{1pt}}
\newcommand{\Sr}{ {\rm S}}
\newcommand{\U}{ {\rm U}}
\newcommand{\BL}{ {\overline L}}
\newcommand{\BE}{ {\overline E}}
\newcommand{\BP}{ {\overline P}}
\newcommand{\Ab}{\mathbb{A}\tss}
\newcommand{\CC}{\mathbb{C}\tss}
\newcommand{\KK}{\mathbb{K}\tss}
\newcommand{\QQ}{\mathbb{Q}\tss}
\newcommand{\SSb}{\mathbb{S}\tss}
\newcommand{\ZZ}{\mathbb{Z}\tss}
\newcommand{\X}{ {\rm X}}
\newcommand{\Y}{ {\rm Y}}
\newcommand{\Z}{{\rm Z}}
\newcommand{\Ac}{\mathcal{A}}
\newcommand{\achi}{\Ac_{\chi}}
\newcommand{\bachi}{\overline\Ac_{\chi}}
\newcommand{\Lc}{\mathcal{L}}
\newcommand{\ol}{\overline}
\newcommand{\Pc}{\mathcal{P}}
\newcommand{\Qc}{\mathcal{Q}}
\newcommand{\Tc}{\mathcal{T}}
\newcommand{\Sc}{\mathcal{S}}
\newcommand{\Bc}{\mathcal{B}}
\newcommand{\Dc}{\mathcal{D}}
\newcommand{\Ec}{\mathcal{E}}
\newcommand{\Oc}{\mathcal{O}}
\newcommand{\Fc}{\mathcal{F}}
\newcommand{\Hc}{\mathcal{H}}
\newcommand{\Uc}{\mathcal{U}}
\newcommand{\Vc}{\mathcal{V}}
\newcommand{\Wc}{\mathcal{W}}
\newcommand{\Yc}{\mathcal{Y}}
\newcommand{\Mc}{\mathcal{M}}
\newcommand{\Kc}{\mathcal{K}}
\newcommand{\Ar}{{\rm A}}
\newcommand{\Br}{{\rm B}}
\newcommand{\Ir}{{\rm I}}
\newcommand{\Fr}{{\rm F}}
\newcommand{\Jr}{{\rm J}}
\newcommand{\Or}{{\rm O}}
\newcommand{\GL}{{\rm GL}}
\newcommand{\Spr}{{\rm Sp}}
\newcommand{\Rr}{{\rm R}}
\newcommand{\Zr}{{\rm Z}}
\newcommand{\gl}{\mathfrak{gl}}
\newcommand{\middd}{{\rm mid}}
\newcommand{\ev}{{\rm ev}}
\newcommand{\Pf}{{\rm Pf}}
\newcommand{\Norm}{{\rm Norm\tss}}
\newcommand{\oa}{\mathfrak{o}}
\newcommand{\spa}{\mathfrak{sp}}
\newcommand{\osp}{\mathfrak{osp}}
\newcommand{\g}{\mathfrak{g}}
\newcommand{\h}{\mathfrak h}
\newcommand{\n}{\mathfrak n}
\newcommand{\z}{\mathfrak{z}}
\newcommand{\Zgot}{\mathfrak{Z}}
\newcommand{\p}{\mathfrak{p}}
\newcommand{\sll}{\mathfrak{sl}}
\newcommand{\agot}{\mathfrak{a}}
\newcommand{\qdet}{ {\rm qdet}\ts}
\newcommand{\Ber}{ {\rm Ber}\ts}
\newcommand{\HC}{ {\mathcal HC}}
\newcommand{\cdet}{ {\rm cdet}}
\newcommand{\tr}{ {\rm tr}}
\newcommand{\gr}{ {\rm gr}}
\newcommand{\str}{ {\rm str}}
\newcommand{\loc}{{\rm loc}}
\newcommand{\Gr}{{\rm G}}
\newcommand{\sgn}{ {\rm sgn}\ts}
\newcommand{\ba}{\bar{a}}
\newcommand{\bb}{\bar{b}}
\newcommand{\bi}{\bar{\imath}}
\newcommand{\bj}{\bar{\jmath}}
\newcommand{\bk}{\bar{k}}
\newcommand{\bl}{\bar{l}}
\newcommand{\hb}{\mathbf{h}}
\newcommand{\Sym}{\mathfrak S}
\newcommand{\fand}{\quad\text{and}\quad}
\newcommand{\Fand}{\qquad\text{and}\qquad}
\newcommand{\For}{\qquad\text{or}\qquad}
\newcommand{\OR}{\qquad\text{or}\qquad}

\renewcommand{\theequation}{\arabic{section}.\arabic{equation}}

\newtheorem{thm}{Theorem}[section]
\newtheorem{lem}[thm]{Lemma}
\newtheorem{prop}[thm]{Proposition}
\newtheorem{cor}[thm]{Corollary}
\newtheorem{conj}[thm]{Conjecture}
\newtheorem*{mthm}{Main Theorem}
\newtheorem*{mthma}{Theorem A}
\newtheorem*{mthmb}{Theorem B}

\theoremstyle{definition}
\newtheorem{defin}[thm]{Definition}

\theoremstyle{remark}
\newtheorem{remark}[thm]{Remark}
\newtheorem{example}[thm]{Example}

\newcommand{\bth}{\begin{thm}}
\renewcommand{\eth}{\end{thm}}
\newcommand{\bpr}{\begin{prop}}
\newcommand{\epr}{\end{prop}}
\newcommand{\ble}{\begin{lem}}
\newcommand{\ele}{\end{lem}}
\newcommand{\bco}{\begin{cor}}
\newcommand{\eco}{\end{cor}}
\newcommand{\bde}{\begin{defin}}
\newcommand{\ede}{\end{defin}}
\newcommand{\bex}{\begin{example}}
\newcommand{\eex}{\end{example}}
\newcommand{\bre}{\begin{remark}}
\newcommand{\ere}{\end{remark}}
\newcommand{\bcj}{\begin{conj}}
\newcommand{\ecj}{\end{conj}}

\newcommand{\bal}{\begin{aligned}}
\newcommand{\eal}{\end{aligned}}
\newcommand{\beq}{\begin{equation}}
\newcommand{\eeq}{\end{equation}}
\newcommand{\ben}{\begin{equation*}}
\newcommand{\een}{\end{equation*}}

\newcommand{\bpf}{\begin{proof}}
\newcommand{\epf}{\end{proof}}
\newcommand{\whg}{\widehat{\mathfrak{g}}}   

\def\beql#1{\begin{equation}\label{#1}}

\title{\Large\bf Quantum super Manin matrices}

\author{{Naihuan Jing,\ \ Yinlong Liu,\ \  Jian Zhang\ \  }}

\date{} 
\maketitle

\vspace{4 mm}

\begin{abstract}
Manin matrices are quantum linear transformations of general quantum spaces.
In this paper, we study the $q$-analogue of super Manin matrices and
obtain several quantum versions of classical identities, such as
Jacobi's ratio theorem, Schur's complement theorem, Cayley's complementary theorem, Muir's law, Sylvester's theorem,
MacMahon Master Theorem and Newton's identities.

\end{abstract}


\vspace{5 mm}

%


\section{Introduction}
\setcounter{equation}{0}

The theory of matrices with noncommutative entries has played a significant role in the development of quantum groups and integrable systems. Yuri Manin introduced a particularly fruitful class of such matrices in the late 1980s ~\cite{Ma1,Ma2}. These are defined as noncommutative quadratic algebras generalizing linear maps for vector spaces, where
quadratic algebras replace the vector spaces with specific commutation relations. The corresponding homomorphisms between such algebras are described by matrices whose entries satisfy certain commutation relations, and this defines the so-called  Manin matrices. Since their inception, Manin matrices have found important applications in quantum groups and
quantum integrable systems, e.g.,  Talalaev's quantum spectral curve formula~\cite{Ta,CT,CF}.

Manin's original motivation was to understand the structure of the quantum plane, defined by the relation \( yx = qxy \), and its endomorphisms, which form the backbone of a quantum group. This was later extended to \( n \) generators with relations \( x_j x_i = q x_i x_j \) for \( i < j \), leading to the notion of  $q$-Manin matrices. Such matrices exhibit remarkable properties: despite noncommutativity of their entries, they admit a natural generalization of many classical linear algebra constructions, such as determinants, minors, Cayley--Hamilton theorem, Cramer's rule, and Capelli-type identities, etc. ~\cite{CSS, CFR, CFRS}.

A further generalization arises when one considers \emph{pairs} of quadratic algebras whose commutation relations are defined by a pair of idempotents \( (A, B) \), where the homomorphisms between them are encoded by matrices satisfying commutation relations given by the internal hom algebra~\cite{Si}. This leads to the concept of \emph{$(A, B)$-Manin matrices}, where \( A \) and \( B \) are idempotents defining the quadratic algebras. When \( A = B \), one recovers the classical case of endomorphisms. This broader framework allows for a unified treatment of various deformed matrix structures, including multi-parametric Manin matrices~\cite{JLZ} and Molev's type \( B, C, D \) Manin matrices~\cite{Mo2, Si1}.

In the super-algebraic setting, Manin also introduced multi-parametric quantum deformations of Lie group \cite{Ma3}
together with its coactions on quantum symmetric and exterior superalgebra.
Hai et al. \cite{HKL} gave some examples of homogeneous superalgebras whose generator matrices are super Manin matrix. Later, Molev and Ragoucy \cite{MR} proved super analogues of the MacMahon Master Theorem and Newton's identities for super Manin matrices and constructed higher order Sugawara operators for the affine Lie superalgebra $\widehat{\mathfrak{gl}}_{m|n}$.

In this paper, we introduce the notion of
$q$-super Manin matrices by contructing the bialgebra associated with the quantum symmetric (resp. exterior) superalgebra, which incorporate both $q$-deformations and super-algebraic structures. Our approach is based on the tensor formalism developed in~\cite{Ma1,Ma2}, where quadratic (super)algebras are defined by idempotents in tensor spaces. We show that these matrices admit natural analogues of determinants and minors, which are defined as quantum Berezinian of q-super Manin matrices.
Moreover, we derive $q$-super-analogs of  Jacobi's ratio theorem, Schur's complement theorem, Cayley's complementary theorem, Muir's law and Sylvester's theorem by using the quasideterminant.

\section{Quantum super Manin matrices}
\setcounter{equation}{0}
We will introduce the $q$-super Manin matrices as the generator matrix of the bialgebra associated with the quantum symmetric superalgebra of a $\ZZ_2$-graded vector space. In the super case with $q=1$, this bialgebra specializes to the right quantum superalgebra defined by Molev and Ragoucy \cite{MR}.
This formation can be viewed as a generalization of the Faddeev-Reshetikhin-Tacktajan formulism of the quantum group \cite{RTF}
in the super case.

Let $\CC^{m|n}$  be the $\ZZ_2$-graded vector space with the basis $e_1,\ldots,e_{m+n}$ such that the parity of $e_i$ is defined as $\bar{e}_i=\bar{i}$ and set
\[
\bar{i}=\left\{\begin{array}{cc}
0&\text{ if }i\leq m,\\
1&\text{ if }i> m.
\end{array}\right.
\]

Let $q_i=q^{1-2\bar i}$ for $i\in [m+n]=\{1,\ldots,m+n\}$. In the superalgebra $\End(\CC^{m|n})^{\ot r}$, $r\geq 2$, for any $1 \leq a \leq r-1$ we define
\begin{equation}\label{P}
  P^q =\sum_{i,j=1}^{m+n}q^{\ve(i-j)}(-1)^{\bar{j}} E_{ij}\ot E_{ji}  .
\end{equation}
where $E_{ij}$ are the standard basis of $\mathrm{End}(\mathbb{C}^{m|n})$ with the parity $\bar i +\bar j$
and the symbol $\ve$  is defined by
\ben
\ve (i-j)
= \left\{ \begin{aligned}
1\ \ &\text{ if }\ i>j,\\
0\ \ &\text{ if }\ i=j,\\
-1\ \ &\text{ if }\ i<j.\\
\end{aligned} \right.
\een
Then
\beq
\bal
(P^q )^{2}&=1,\\
P^q_{12}P^q_{23}P^q_{12}&=P^q_{23}P^q_{12}P^q_{23}.
\eal
\eeq

Let    $S^q(V)=T(V)/R$  be the quantum symmetric superalgebra  generated by $x_1, \ldots, x_{m+n}$ with parity $\bar{x}_i=\bar{i}$  satisfying the quadratic relation:
\beq\label{symmetric-superalg}
(1-P^q)X\ot X=0,
\eeq
where $X=(x_1,\ldots,x_{m+n})^t$.
Explicitly the quadratic relation \eqref{symmetric-superalg} is written as
\beq\label{symmetric-rela}
x_i^2=0, \ \text{if} \ \bar{i}=1; \ \ \ x_j x_i-q^{\ve(j-i)}(-1)^{\bar{i}\bar{j}}x_ix_j=0.
\eeq

Define the quantum exterior superalgebra $\Lambda^q(V^*)=T(V^*)/R'$ on the dual space $V^*$ is generated by $\psi_{1},\ldots,\psi_{m+n}$ with  parity $\bar{\psi}_i=1-\bar{i}$ satisfying the quadratic relation:
\beq\label{exterior-superalg}
(\Psi\otimes\Psi)(1+P^q)=0,
\eeq
where  $ \Psi=(\psi_1, \ldots, \psi_{m+n})$. This quadratic relation \eqref{exterior-superalg} can be written entry-wise as
\beq\label{exterior-rela}
\psi_i^2=0, \ \text{if} \ \bar{i}=0; \ \ \ \psi_k \psi_l-q^{\ve(l-k)}(-1)^{(\bar{k}+1)(\bar{l}+1)}\psi_l \psi_k=0.
\eeq
The two quadratic algebras form a dual pair with a bilinear form $\langle \ |\ \rangle:\Lambda^q(V^*) \ot S^q(V) \rightarrow \CC$ by $\langle  \psi_{j} | x_i \rangle= \delta_{ij}$ and $\langle  \psi_{k}\otimes \psi_{l} | x_i \otimes x_j \rangle= (-1)^{\bar{i}(\bar{l}+1)} \delta_{ik}\delta_{jl}$.


From the construction of bialgebra associated with an arbitrary quadratic algebra or superalgebra given by Manin \cite{Ma1,Ma2},
we can define a superalgebra $\mathcal{M}^q_{m|n}=T(V^*\ot V)/\langle \tau_{23}(R'\ot R)\rangle$ called the {\it right quantum superalgebra}, where $\tau(a\ot b)=(-1)^{\bar{a}\bar{b}}b \ot a$, is generated by $M_{ij}, 1 \leq i,j \leq m+n$ with parity $\bar{i}+\bar{j}$ and satisfies the quadratic relations:
\beq\label{M-rel}
\bal
&M_{ij}^2=0, \ \text{if} \ \bar{i}=1, \bar{j}=0;\\
&M_{ik}M_{il}-q^{\ve(l-k)}(-1)^{(\bar{k}+1)(\bar{l}+1)}M_{il}M_{ik}=0 , \ \text{if} \ \bar{i}=1;\\
&M_{ik}M_{jk}-q^{\ve(i-j)}(-1)^{\bar{i}\bar{j}}M_{jk}M_{ik}=0 , \ \text{if} \ \bar{k}=0;\\
&M_{ij}M_{kl}-q^{\ve(i-k)+\ve(l-j)}(-1)^{(\bar{i}+\bar{j})(\bar{k}+\bar{l})}M_{kl}M_{ij} \\
&= q^{\ve(i-k)}(-1)^{\bar{i}\bar{k}+\bar{j}\bar{k}+\bar{i}\bar{j}}M_{kj}M_{il}-
q^{\ve(l-j)}(-1)^{\bar{j}\bar{k}+\bar{j}\bar{l}+\bar{k}\bar{l}}M_{il}M_{kj} ,\ \ 1 \leq i,j,k,l \leq m+n.
\eal
\eeq
For simplicity we usually refer $\mathcal{M}^q_{m|n}$ as $\mathcal{M}_{m|n}$ if $q$ is clear from the context.
The algebra  $\mathcal{M}_{m|n}$ is a bialgebra under the comultiplication
 $\mathcal{M}_{m|n}\rightarrow \mathcal{M}_{m|n}\ot \mathcal{M}_{m|n}$
 defined by
\begin{equation}
\Delta(M_{ij})=\sum_{k=1}^{m+n} M_{ik}\otimes M_{kj},
\end{equation}
and the counit given by $\varepsilon(M_{ij})=\delta_{ij}$.

Let $M=(M_{ij})$ be the $(m+n)\times (m+n)$ generator matrix. Define $Y=MX$ and $\Phi=\Psi M$, i.e.
\begin{equation}\label{e:y-phi}
y_i=\sum_{j=1}^{m+n}M_{ij}x_j, \qquad \phi_j=\sum_{i=1}^{m+n}\psi_iM_{ij}.
\end{equation}
\begin{prop}\label{Manin-equi-rela}
The quadratic relations \eqref{M-rel} of the bialgebra $\Mc_{m|n}$ are  equivalent to one of
the following three conditions:
\begin{align}\label{relation manin matrix}
&(1-P^q) M_{1}M_{2}(1+P^q)=0,\\
&(1-P^q)(Y\otimes Y)=0,\\
&(\Phi\otimes\Phi)(1+P^q)=0,
\end{align}
where $M_1=\sum_{i,j}(-1)^{\bar{i}\bar{j}+\bar{j}} M_{ij}\ot E_{ij} \ot 1$, $M_2=\sum_{i,j}(-1)^{\bar{i}\bar{j}+\bar{j}} M_{ij} \ot 1 \ot E_{ij}$.
\end{prop}
Any $(m+n)\times (m+n)$ matrix $M$  satisfying the equation \eqref{relation manin matrix} is called a {\it q-super Manin matrix}. If $q=1$, then $M$
 reduces to the super Manin matrix studied in \cite{MR}. When $n=0$ and $q=1$, the formulation of Manin matrices was given by \cite{Si}.
\begin{prop}
There exist algebra morphisms
  \beq
  \delta: S^q(V)\rightarrow\Mc_{m|n} \ot S^q(V), \ \ \ \ \delta^*: \Lambda^q(V^*)\rightarrow \Lambda^q(V^*) \ot \Mc_{m|n}
  \eeq
  such that
  \beq
  \delta(X)= M\ot X\ \ \text{i.e.} \ \ \delta(x_i)=\sum_{j=1}^{m+n} M_{ij}\ot x_j;
  \eeq

  \beq
  \delta^*(\Psi)=\Psi\ot M \ \  \text{i.e.} \ \ \delta^*(\psi_i)=\sum_{j=1}^{m+n}\psi_j\ot M_{ji}.
  \eeq
\end{prop}

 The symmetric group $\mathfrak{S}_{k}$ acts on the space $(\mathbb{C}^{m|n})^{\otimes k}$ via
$s_{i}\mapsto P_{s_{i}}^{q} :=P_{i,i+1}^{q}$ for $i=1, \ldots, k-1$, where $s_{i}$ denotes
the adjacent transpositions $(i,i+1)$. If $\sigma=s_{i_{1}}\cdots s_{i_{l}}$ is a reduced decomposition of
an element $\sigma\in \mathfrak{S}_{k}$, we set $P_{\sigma}^{q}= P_{s_{i_{1}}}^{q}\cdots P_{s_{i_{l}}}^{q}$.
The $q$-antisymmetrizer and $q$-symmetrizer are then defined by
\beq
	\mathcal{A}_{k}^q =\frac{1}{k!} \sum_{\sigma\in \mathfrak{S}_{k}}\mathrm{sgn}\sigma\cdot P_{\sigma}^{q}\quad\text{ and }\quad
	\mathcal{H}_{k}^q =\frac{1}{k!} \sum_{\sigma\in \mathfrak{S}_{k}}P_{\sigma}^{q}\quad\text{ respectively.}
\eeq
Then any $(m+n)\times (m+n)$ matrix $M$  is a {\it $q$-super Manin matrix} if and only if it satisfies the following relation:
\beq
\mathcal{A}_{2}^qM_1M_2\mathcal{H}_{2}^q=0.
\eeq

 For an $(m+n)\times(m+n)$-matrix $A=(A_{ij})$ where $A_{ij}\in \Mc_{m|n}$, we define $A_i\in \Mc_{m|n} \ot \End(\CC^{m|n})^{\ot k}$ as follows:
\begin{equation}
A_i=\sum_{i,j}(-1)^{\bar{i}\bar{j}+\bar{j}} A_{ij}\ot 1^{\ot(i-1)}\ot E_{ij} \ot 1^{\ot(k-i)}.
\end{equation}
 Then we have the following identities in the superalgebra $\Mc_{m|n} \ot \End(\CC^{m|n})^{\ot k}$:

\begin{thm}\label{comm} For the generator matrix $M$ one has that
  \begin{align}
    \Ac_k^q M_1\ldots M_k \Ac_k^q=\Ac_k^q M_1\ldots M_k,\\
    \Hc_k^q M_1 \ldots M_k \Hc_k^q=M_1\ldots M_k \Hc_k^q.
  \end{align}
\end{thm}
\begin{proof}
To prove the first identity, it suffices to show that for any $a=1,
\ldots,k-1$, we have
\beq
 \Ac_k^q M_1\ldots M_k P^q_{a,a+1}=-\Ac_k^q M_1\ldots M_k.
\eeq
Since $\Ac_k^q=\Ac_k^q (1-P_{a,a+1}^q)/2$, it is enough to consider the case $k = 2$.
And the relation when $k=2$ can be deduced  from \eqref{relation manin matrix}.  Similarly, we can prove the second identity from \eqref{relation manin matrix}.
\end{proof}
For  a matrix $A=\sum _{i,j=1}^{m+n}a_{ij}E_{ij}$,   the supertranspose is defined by
$$A^{st}=\sum_{i,j}(-1)^{\bar{i}(\bar{i}+\bar{j})}a_{ji}E_{ij},$$
and the supertrace $str$ by
\[
str(A)=\sum_{i=1}^{m+n}(-1)^{\bar{i}}a_{ii}.
\]
For any  $a\in\{1,2,\ldots,k\}$ we will denote by $str_a$ the corresponding partial transposition on the
algebra $(\End \mathbb C^{m|n})^{\otimes k}$ which acts as $str$ on the $a$-th copy of $\End \mathbb C^{m|n}$ and as the identity map on all
the other tensor factors.

Employing  the approach of  \cite{MR}, we obtain the MacMahon Master Theorem for the superalgebra $\Mc_{m|n}$.
\begin{thm}
For any $k\geq 1$, we have the identities in the superalgebra $\Mc_{m|n}$
\begin{align}
  \sum_{r=0}^k(-1)^r str_{1,\ldots,k}\Hc_r^q\Ac_{\{r+1,\ldots,k\}}^qM_1\ldots M_k=0,\\
  \sum_{r=0}^k (-1)^r  str_{1,\ldots,k}\Ac_r^q\Hc_{\{r+1,\ldots,k\}}^qM_1\ldots M_k=0,
\end{align}
where $\Ac_{\{r+1,\ldots,k\}}^q$ and $\Hc_{\{r+1,\ldots,k\}}^q$ denote the $q$-antisymmetrizer and $q$-symmetrizer over the copies of  $\End(\CC^{m|n})$ labeled by $r+1,\ldots,k$.
\end{thm}
\begin{proof}
We first prove the first equation by Theorem \ref{comm} and the following facts:
\begin{equation}
\begin{aligned}
&(k-r+1)\Ac_{\{r,\ldots,k\}}^q=\Ac_{\{r+1,\ldots,k\}}^q-(k-r)\Ac_{\{r+1,\ldots,k\}}^q  P_{r,r+1}^{q}\Ac_{\{r+1,\ldots,k\}}^q,\\
&(r+1)\Hc_{r+1}^q=\Hc_r^q+r\Hc_r^q P_{r,r+1}^{q} \Hc_r^q.
\end{aligned}
\end{equation}
Hence we have that
\begin{equation}\label{trace replacement}
\begin{aligned}
&str_{1,\ldots,k}\Hc_r^q\Ac_{\{r+1,\ldots,k\}}^qM_1\ldots M_k\\
=&str_{1,\ldots,k}\frac{r(k-r+1)}{k}\Hc_r^q\Ac_{\{r,\ldots,k\}}^q
M_1\cdots M_k\\
+
&str_{1,\ldots,k}\frac{(r+1)(k-r)}{k}\Hc_{r+1}^q\Ac_{\{r+1,\ldots,k\}}^q
M_1\cdots M_k.
\end{aligned}
\end{equation}
Therefore the telescoping sum equals to zero. The second equation can be proved by the same arguments.

\end{proof}

Let $B, C$ be $(m+n)\times (m+n)$ matrices.
Denote $B*C=str_1 P^q_{12} B_1C_2$. Let $M^{[k]}$ be the $k$-th power of $M$ under the multiplication $*$, i.e.
\begin{align}\label{e:power-m}
M^{[0]}=1, M^{[1]}=M, M^{[k]}=M^{[k-1]}*M,k>1.
\end{align}

We denote $A(t)$, $S(t)$ and $T(t)$ as follows:
\begin{align}
S(t)&=\sum_{k=0}^{\infty}t^k str_{1,\ldots,k}\Hc_k^q M_1\ldots M_k,\\
A(t)&=\sum_{k=0}^{\infty}(-t)^k str_{1,\ldots,k}\Ac_k^q M_1\ldots M_k,\\
T(t)&=\sum_{k=0}^{\infty}t^k strM^{[k+1]}.
\end{align}

We now have Newton's identities for $q$-super Manin matrix.
\begin{thm}[Newton's identities]
Let $M$ be an $(m+n)\times (m+n)$ $q$-super Manin matrix.  Then
\begin{align}
&\partial_t A(t)=-A(t) T(t),\\
&\partial_t S(t)= T(t)S(t).
 \end{align}
\end{thm}
\begin{proof}
By the fact
\beq
k\Ac_{k}^q=\Ac_{k-1}^q-(k-1)\Ac_{k-1}^q  P_{k-1,k}^q \Ac_{k-1}^q,
\eeq
we have that
\beq
\bal
&kstr_{1,\ldots,k-1}\Ac_k^q M_1\ldots M_k\\
&= str_{1,\ldots,k-1}\Ac_{k-1}^q M_1\ldots M_{k-1}M_k-(k-1)str_{1,\ldots,k-1}\Ac_{k-1}^q  P_{k-1,k}^q \Ac_{k-1}^q M_1\ldots M_k\\
&=(str_{1,\ldots,k-1}\Ac_{k-1}^q M_1\ldots M_{k-1})M-(k-1)str_{1,\ldots,k-2}(\Ac_{k-1}^q  M_1\ldots M_{k-1})*M\\
&\ldots\\
&=\sum_{i=0}^{k-1}(-1)^{k+i+1} (str_{1,\ldots,i}\Ac_i^q M_1\ldots M_i) M^{[k-i]}.
\eal
\eeq
Taking the supertrace on both sides of the above equation, we have that
\begin{equation}\label{eq cayleyhamiton}
kstr_{1,\ldots,k}\Ac_k^q M_1\ldots M_k=\sum_{i=0}^{k-1}(-1)^{k+i+1} (str_{1,\ldots,i}\Ac_i^q M_1\ldots M_i)\cdot str M^{[k-i]}.
\end{equation}
This implies that $\partial_t A(t)=-A(t) T(t)$.

It follows from the MacMahon Master Theorem that
$A(t)S(t)=S(t)A(t)=1$.
By the Leibniz rule, we have
\begin{align}
(\partial_t A(t))  S(t) +A(t)\partial_t S(t)=0.
\end{align}

Therefore, $A(t)\partial_t S(t)=A(t) T(t)S(t)$. Multiplying the both sides of the equation by $S(t)$ from the left, we have
\begin{equation}
\partial_t S(t)= T(t)S(t).
\end{equation}
\end{proof}


Let $I=(i_1,\ldots,i_k)$, $J=(j_1,\ldots,j_k)$, and $\alpha_i(I)$ denote the multiplicity of index $i$ in the given multiset $I$ of $[m+n]$. Given any permutation $\si \in \Sym_k$,
\beq\label{P-act1}
P_{\si}^q(E_{i_1j_1}\ot \ldots \ot E_{i_kj_k} )= \varepsilon(\si, I,J)E_{i_{\si^{-1}(1)}j_1}\ot \ldots \ot E_{i_{\si^{-1}(k)}j_k},
\eeq
where the coefficient $\varepsilon(\si, I,J)$ is denoted by
$$\varepsilon(\si, I,J)=\prod_{k<t, \atop \si(k)> \si(t)}q^{\ve(i_t-i_k)}(-1)^{\bar i_k \bar i_t+\bar j_k(\bar i_k +\bar i_t)}.$$ Similarly,
\beq\label{P-act2}
(E_{i_1j_1}\ot \ldots \ot E_{i_kj_k} )P_{\si^{-1}}^q= \omega(\si, I,J)E_{i_1 j_{\si^{-1}(1)}}\ot \ldots \ot E_{i_k j_{\si^{-1}(k)}},
\eeq
where the coefficient $\omega(\si, I, J)$ is denoted by
$$\omega(\si, I, J)=\prod_{k<t, \atop \si(k)> \si(t)}q^{\ve(j_k-j_t)}(-1)^{\bar j_k \bar j_t+\bar i_t(\bar j_k +\bar j_t)}.$$
We also set
\beq\label{gamma}
\gamma(I,J)=\sum_{a} \bar i_a\bar j_a+\sum_{a<b}  (\bar i_a+ \bar j_a)( \bar i_b +\bar j_b).
\eeq

\begin{prop}
We have the explicit formulas in superalgebra $\Mc_{m|n}$
\beq
\bal
str_{1,\ldots,k}\Ac_k^q M_1\ldots M_k=& \sum_{I} \frac{1}{\alpha_{m+1}!\ldots \alpha_{m+n}!} \\
&\times \sum_{\si\in \Sym_k} \sgn \si \cdot \varepsilon(\si,\si I, I)M_{i_{\si(1)}i_1}\ldots M_{i_{\si(k)}i_k}(-1)^{\gamma(\si I,I)},
\eal
\eeq
where the sum is over multisets  $I=\{i_1 < \ldots < i_l < m+1 \leq i_{l+1} \leq \ldots \leq  i_k \}$ with $l=0,\ldots ,k$, where $\si I=\{i_{\si(1)},\ldots,i_{\si(k)}\}$. Moreover,
\beq
\bal
str_{1,\ldots,k} M_1\ldots M_k\Hc_k^q=& \sum_{I} \frac{1}{\alpha_{1}!\ldots \alpha_{m}!} \\
&\times \sum_{\si\in \Sym_k} \sgn \si \cdot \omega(\si, I,\si I)M_{i_1 i_{\si(1)}}\ldots M_{i_k i_{\si(k)}}(-1)^{\gamma( I,\si I)},
\eal
\eeq
where the sum is over multisets  $I=\{i_1 \leq \ldots \leq i_l \leq m <i_{l+1}<\ldots < i_k \}$ with $l=0,\ldots ,k$.
\end{prop}
\begin{proof}
  Write
  \ben
  \Ac_k^qM_1\ldots M_k= \sum_{I,J}M_{j_1\ldots j_k}^{i_1 \ldots i_k}\ot E_{i_1j_1}\ot \cdots \ot E_{i_k j_k},
  \een
  summed over multisets $I=\{i_1,\ldots,i_k\}$ and $J=\{j_1,\ldots,j_k\}$, where $ M_{j_1\ldots j_k}^{i_1 \ldots i_k}\in \Mc_{m|n}$. By the theorem \ref{comm}, we have that
  \beq
     P_{a,a+1}^q\Ac_k^qM_1\ldots M_k=\Ac_k^qM_1\ldots M_k P_{a,a+1}^q=-\Ac_k^qM_1\ldots M_k.
  \eeq
Therefore,
\beq
  M_{j_1\ldots j_k}^{i_1 \ldots i_{a+1} i_a \ldots i_k}=-q^{\ve(i_{a+1}-i_a)}M_{j_1\ldots j_k}^{i_1 \ldots i_k}(-1)^{\bar i_a \bar i_{a+1} +\bar i_a \bar j_a + \bar i_{a+1}\bar j_a}
\eeq
and
\beq
  M^{i_1\ldots i_k}_{j_1 \ldots j_{a+1} j_a \ldots j_k}=-q^{\ve(j_a-j_{a+1})}M^{i_1\ldots i_k}_{j_1 \ldots j_k}(-1)^{\bar j_a \bar j_{a+1} +\bar j_a \bar i_{a+1} + \bar j_{a+1}\bar i_{a+1}}
\eeq
Either of these relations implies that
if $i_a=i_{a+1}$,
\beq
  M_{i_1\ldots i_k}^{i_1 \ldots i_k}=-M_{i_1\ldots i_k}^{i_1 \ldots i_k}(-1)^{\bar i_a}
\eeq
so that the coefficient vanishes  if $\bar i_a=0$.
Moreover,
\beq
  M_{i_1\ldots i_{a+1} i_a \ldots i_k}^{i_1\ldots i_{a+1} i_a \ldots i_k}=M_{i_1\ldots i_k}^{i_1 \ldots i_k}.
\eeq
 By the definition of the supertrace, we have
 \beq
 str_{1,\ldots,k}\Ac_k^q M_1\ldots M_k=\sum_I \frac{k!}{\alpha_{m+1}!\ldots \alpha_{m+n}!}M_{i_1\ldots i_k}^{i_1 \ldots i_k}(-1)^{\bar i_1+\cdots +\bar i_k},
 \eeq
 where the sum is over multisets  $I=\{i_1 < \ldots < i_l < m+1 \leq i_{l+1} \leq \ldots \leq  i_k \}$ with $l=0,\ldots ,k$.

On the other hand, using the notations \eqref{P-act1} and \eqref{gamma} we can write
\beq
M_{j_1\ldots j_k}^{i_1 \ldots i_k}=\frac{1}{k!}\sum_{\si \in \Sym_k}\sgn \si  \cdot \varepsilon(\si,\si I, J) M_{i_{\si(1)}j_1}\ldots M_{i_{\si(k)}j_k}(-1)^{\gamma(\si I,J)}(-1)^{\bar j_1+\cdots +\bar j_k}.
\eeq
which completes the proof of the first equation. By the same argument we can prove the second.
\end{proof}

\section{Quantum Berezinian}
\setcounter{equation}{0}
We employ Manin's approach \cite{Ma3} to define the quantum Berezinian of the $q$-super Manin matrix.
Denote $\Kc_{m|n}=\Lambda^q(V^*)\ot S^{q}(V)$. Then $\Kc_{m|n}$ is a superalgebra with $2(m+n)$ generators $\psi_1\ot 1, \ldots, \psi_{m+n}\ot 1$, $1 \ot x_1 \ldots, 1 \ot x_{m+n}$. Let
\ben
c=\sum_{i=1}^{m+n}\psi_i\ot x_i \in \Kc_{m|n}.
\een
It is easy to verify that $c^2=0$. And define the linear map $d$ by
\beq
\bal
d: \Kc_{m|n}&\rightarrow \Kc_{m|n}\\
k &\mapsto k\cdot c
\eal
\eeq
Hence considering the cohomology  with respect to $d$: $$H(\Kc_{m|n})=\mathrm{ker}(d)/\mathrm{Im}(d).$$
Then we have the following proposition.

\begin{prop} \cite{Ma3}
$H(\Kc_{m|n})$ is an one-dimensional vector space spanned by
$$
\prod_{\bar i=0}\psi_i\prod_{\bar j =1}x_j \mod \mathrm{Im}(d).
$$
\end{prop}
This space is the superanalog of the quantum highest exterior power. Now we define a Hopf superalgebra which coacts upon the superspace.

Define $\overline{\Mc}_{m|n}$ is the free associative algebra generated by the entries of $\bar{M}_0, \bar{M}_1, \ldots $ where  these matrices have independent entries.
Define a structure of coalgebra on $\overline{\Mc}_{m|n}$ by
$$\bar{\Delta}(\bar{M}_k^{st^{k}})=\bar{M}_k^{st^{k}}\ot \bar{M}_k^{st^{k}}, \ \varepsilon(\bar{M}_k)=I, \ k=0,1,2,\ldots $$
And define a linear map $\imath:\overline{\Mc}_{m|n} \rightarrow \overline{\Mc}_{m|n} $ by
\begin{align}
   \imath(\bar{M}_k)=\bar{M}_{k+1}, \ \  \imath(uv)=(-1)^{\bar u \bar v}\imath(v)\imath(u).
\end{align}
Denote by $R^q=\langle \tau_{23}(R'\ot R)\rangle$ (resp. $R^{q^{-1}}$) the ideal in $\Mc_{m|n}^q$ (resp. $\Mc_{m|n}^{q^{-1}}$) generated by the left sides of relations \eqref{M-rel} (resp. replacing  $q$ with $q^{-1}$ in relations \eqref{M-rel}).

Define $\widetilde{\Mc}_{m|n}$ is the  quotient algebra of $\overline{\Mc}_{m|n}$ by the ideal $\bar{R}$ generated by
\begin{align}
    \text{the entries of } \bar{M}_{k}^{st^{k}}\bar{M}_{k+1}^{st^{k}}-I, \ \bar{M}_{k+1}^{st^{k}} \bar{M}_{k}^{st^{k}}-I, \ &k=0,1,2,\ldots;\\
    R_k=\imath^k(R^q), \ &k=0,1,2,\ldots ,
\end{align}
where $R_k$ is the  relation satisfied by $\bar{M}_k$ in the quotient algebra  $\widetilde{\Mc}_{m|n}$.
Then there exists an embedding $\gamma:\Mc_{m|n}^q \rightarrow \widetilde{\Mc}_{m|n} $ defined by
\beq
\gamma(M)=\bar{M}_0.
\eeq
The quotient algebra $\widetilde{\Mc}_{m|n}$ is a Hopf superalgebra induced from Hopf superalgebra $\overline{\Mc}_{m|n}$ with an antipode $\imath$.


We impose additional relations $\bar{M}_{k}=\bar{M}_{k+2d},\ d\in \mathbb{Z}^{+}$. Note that $\imath(R^q)=R^{q^{-1}}$.  For convenience, we will denote  $\bar{M}_{k}\ (k=2d)$ by $M$,  $\imath(M)=M^{-1}$.
Define the    right coacting of  $\widetilde{\Mc}_{m|n}$ on $\Kc_{m|n}$ by
\beq\label{Hopf-coaction}
\bal
\delta_{\Kc}: \Kc_{m|n} &\rightarrow\Kc_{m|n} \ot \widetilde{\Mc}_{m|n}\\
\delta_{\Kc}(\Psi) &= \Psi \ot  M,\\
\delta_{\Kc}(X^t)&= X^t \ot \imath(M^{st}),
\eal
\eeq
then we extend the coaction on superalgebra $\Kc_{m|n}$ by the algebra structure of $\Lambda^q(V^*)$ and opposite algebra structure of $S^q(V)$.
Then we have that
\beq
\bal
\delta_{\Kc}(c)&=\sum_{i=1}^{m+n}\delta_{\Kc}(\psi_i)\delta_{\Kc}(x_i)\\
&=\sum_{i=1}^{m+n} (\sum_{j=1}^{m+n} \psi_j \ot M_{ji})\cdot(\sum_{k=1}^{m+n}x_k \ot (M^{-1})^{st}_{ki})\\
&=\sum_{j,k=1}^{m+n} \sum_{i=1}^{m+n} (-1)^{(\bar i+\bar j)\bar k} \psi_j\ot x_k \ot  M_{ji}(M^{-1})^{st}_{ki} \\
&=\sum_{j,k=1}^{m+n} \sum_{i=1}^{m+n} (-1)^{\bar k \bar j + \bar k} \psi_j\ot x_k \ot  M_{ji}(M^{-1})_{ik} \\
&=\sum_{j=1}^{m+n} \psi_j\ot x_j\ot 1\\
&=c\ot 1 .
\eal
\eeq
Hence, $c$ is a coinvariant under $\delta_{\Kc}$. Then we can define the coaction of $\widetilde{\Mc}_{m|n}$ on $H(\Kc_{m|n})$ by
\beq
\bal
\delta_{H}: H(\Kc_{m|n})\rightarrow  H(\Kc_{m|n})\ot \widetilde{\Mc}_{m|n} .
\eal
\eeq
Then there exists some element $\Ber_q(M)\in \widetilde{\Mc}_{m|n}$ such that
\beq
\delta_{H}(\mathop{\prod}\limits^{\rightarrow}_{\bar i=0}\psi_i \mathop{\prod}\limits^{\leftarrow}_{\bar j =1}x_j )=\mathop{\prod}\limits^{\rightarrow}_{\bar i=0}\psi_i \mathop{\prod}\limits^{\rightarrow}_{\bar j =1}x_j  \ot \Ber_q(M),
\eeq
We call $\Ber_q(M)$ the {\it quantum Berezinian} of $M$. Explicitly,
\beq\label{Bere}
\bal
\Ber_q(M)=& \sum_{\si \in \Sym_m}(-q)^{-l(\si)}M_{\si(1),1}M_{\si(2),2}\cdots M_{\si(m),m}\\
&\times \sum_{\rho \in \Sym_n}(-q)^{-l(\rho)}(M^{-1})_{m+1,m+\rho(1)}\cdots (M^{-1})_{m+n,m+\rho(n)}.
\eal
\eeq


We can also define the {\it quantum Berezinian} $\Ber_{q^{-1}}(M^{-1})$ of the  $\qin$-super Manin matrix $M^{-1}$ by defining the right coaction of $\widetilde{\Mc}_{m|n}^{\text{cop}}$ on $\Kc_{m|n}'=\Lambda^{\qin}(V^*)\ot S^{\qin}(V)$ as
\beq
\bal
\delta_{\Kc'}(\Psi) &= \Psi \ot  M^{-1},\\
\delta_{\Kc'}(X^t)&= X^t \ot \imath(M^{-1})^{st},
\eal
\eeq
where superalgebra $\widetilde{\Mc}_{m|n}^{\text{cop}}$ is obtained from $\widetilde{\Mc}_{m|n}$ by taking the opposite of the coalgebra structure.
And we extend the coaction on superalgebra $\Kc'_{m|n}$ by the algebra structure of $\Lambda^{\qin}(V^*)$ and opposite algebra structure of $S^{\qin}(V)$.

Denote by $\{\psi'_i, x'_j \mid 1 \leq i ,j \leq m+n\}$ the generators of  $\Kc_{m|n}'$.
Then
\beq
\bal
\delta_{H'}( \mathop{\prod}\limits^{\leftarrow}_{\bar j=1}x'_j \mathop{\prod}\limits^{\rightarrow}_{\bar i=0}\psi'_i)= \mathop{\prod}\limits^{\rightarrow}_{\bar j =1}x'_j \mathop{\prod}\limits^{\rightarrow}_{\bar i=0}\psi'_i\ot \Ber_{\qin}(M^{-1}).
\eal
\eeq
Explicitly,
\beq\label{Bere-inverse}
\bal
\Ber_{q^{-1}}(M^{-1})=& \sum_{\rho \in \Sym_n}(-q)^{l(\rho)}M_{m+1,m+\rho(1)}\cdots M_{m+n,m+\rho(n)}\\
&\times \sum_{\si \in \Sym_m}(-q)^{l(\si)}(M^{-1})_{\si(1),1}(M^{-1})_{\si(2),2}\cdots (M^{-1})_{\si(m),m}.
\eal
\eeq
Write $M$ in block matrix,
\[
M=\left(\begin{array}{cc}
	M_{11}&M_{12}\\
	M_{21}&M_{22}\\
\end{array}\right)
\]
such that $M_{11}$ and $M_{22}$ are even submatrices  of size $m\times m$ and $n\times n$ respectively. Matrices $M_{12}$ and $M_{21}$ are odd submatrices  of size $m\times n$ and $n\times m$  respectively. We denote by $\pi(A)$ the matrix whose $(i,j)$-th entry is $A_{k+1-i,k+1-j}$ for any  $k \times k$ matrix $A$. Define the $\Pi$-transpose of $M$:
\[
\left(\begin{array}{cc}
	M_{11}&M_{12}\\
	M_{21}&M_{22}\\
\end{array}\right)^{\Pi}
=\left(\begin{array}{cc}
	\pi(M_{22})&\pi(M_{21})\\
	\pi(M_{12})&\pi(M_{11})\\
\end{array}\right).
\]
So $M^{\Pi\circ st}$ is a $q$-super Manin matrix with entries in $\mathcal{M}_{n|m}$ and the coaction of $M^{\Pi\circ st}$ on $\Kc_{n|m}$  is
\beq
\bal
\delta_{H}^{M^{\Pi}}(\Psi) &= \Psi \ot  M^{\Pi\circ st},\\
\delta_{H}^{M^{\Pi}}(X^t)&= X^t \ot \imath (M^{\Pi\circ st^{2}}).
\eal
\eeq
Then the quantum Berezian of $M^{\Pi\circ st}$ is
\beq
\bal
\Ber_q(M^{\Pi})^{st}=& \sum_{\rho \in \Sym_n}(-q)^{-l(\rho)}(M^{\Pi})_{1,\rho(1)}\cdots (M^{\Pi})_{n,\rho(n)} \\
&\times \sum_{\si \in \Sym_m}(-q)^{-l(\si)}((M^{-1})^{\Pi})_{n+\si(1),n+1}\cdots ((M^{-1})^{\Pi})_{n+\si(m),n+m} \\
=&\sum_{\rho \in \Sym_n}(-q)^{-l(\rho)}M_{m+n,m+\rho(n)}\cdots M_{m+1,m+\rho(1)}\\
&\times \sum_{\si \in \Sym_m}(-q)^{-l(\si)}(M^{-1})_{\si(m),m}\cdots (M^{-1})_{\si(1),1}  .
\eal
\eeq
Combination with the relations  $R^q$ and $R^{\qin}$ in  $\widetilde{\Mc}_{m|n}$,
\beq\label{pi-tran-rela}
\bal
\Ber_q(M^{\Pi\circ st}) =\Ber_{q^{-1}}(M^{-1}).
\eal
\eeq

In the following, we consider the ordered multisets of $[m+n]$ which are called the multi-indices.
A multi-index is a tuple $I=(i_1,i_2,\ldots,i_r)$ of integers in $[m+n]$.
The symbol $\Inv$ is defined by
\begin{equation}
\Inv( I)=
\left\{ \begin{aligned}
&\prod_{s<t\atop i_s>i_t}(-q)^{-1} \ &if\ i's\ are\ distinct,\\
&\ 0 \ &if\ two\ i's\ coinside.
\end{aligned} \right.
\end{equation}
\begin{prop}\label{permutation}
For any multi-indices $I=(i_1,i_2,\ldots,i_m)$ and $J=(j_1,j_2,\ldots,j_n)$ of $[m]$ and $[n]$ respectively, one has that
\beq
\bal
\Inv(I)\Inv(J)\Ber_q(M)=&\sum_{\si \in\Sym_m} (-q)^{-l(\si)} M_{\si(1),i_1} M_{\si(2),i_2}\cdots M_{\si(m),i_m}\\
&\times \sum_{\rho \in \Sym_n}(-q)^{-l(\rho)}(M^{-1})_{m+j_1,m+\rho(1)}\cdots (M^{-1})_{m+j_n,m+\rho(n)}.
\eal
\eeq
\end{prop}
\begin{proof}
  Note that
\beq
\delta_{H}(\mathop{\prod}\limits^{\rightarrow}_{1\leq s\leq m}\psi_{i_s} \mathop{\prod}\limits^{\leftarrow}_{1\leq t \leq n}x_{m+j_t} )=\mathop{\prod}\limits^{\rightarrow}_{\bar i=0}\psi_i \mathop{\prod}\limits^{\rightarrow}_{\bar j =1}x_j  \ot \Ber_q'(M).
\eeq
And by the relation \eqref{symmetric-rela} in  $ S^{q}(V)$ and relation \eqref{exterior-rela} in  $\Lambda^q(V^*)$, we have that
\beq
\mathop{\prod}\limits^{\rightarrow}_{1\leq s\leq m}\psi_{i_s} \mathop{\prod}\limits^{\leftarrow}_{1\leq t \leq n}x_{m+j_t}=\Inv(I)\Inv(J)\mathop{\prod}\limits^{\rightarrow}_{\bar i=0}\psi_i \mathop{\prod}\limits^{\leftarrow}_{\bar j =1}x_j .
\eeq
Hence we obtain the proposition.
\end{proof}

We denote by $I\oplus J=(i_1,\ldots,i_{m},j_1,\ldots,j_n)$ the juxtaposition 
of two multi-indices
$I=(i_1,\ldots,i_{m})$, $J=(j_1,\ldots,j_n)$.
Let $I=(i_1,\ldots,i_{r})$, $K=(k_1,\ldots,k_s)$ be multi-indices of $[m+n]$.
We say $I$ is contained in $K$ (still denoted as
$I\subset K$)
if there exist distinct $1', \ldots, r'\in\{1, \ldots, s\}$
such that $i_a=k_{a'}$ for $1\le a\le s$.
Denote by $K\setminus I$
the multi-index
$(k_1,\ldots,\hat{i_1},\ldots,\hat{i_r},\ldots,k_s)$ obtained from
 $(k_1,k_2,\ldots,k_s)$ by deleting $i_1,\ldots,i_r$.
 If $K=[m+n]$ then let $I^c=K \setminus I$ be increasing for any $I$.

Given an increasing multi-index $I_1$ of $[m]$ (resp. $I_2$ of $[m+n]\setminus [m]$ ) with cardinality $r$ (resp. s).  And denote $I=I_1\oplus I_2$ and
$M_{I}$  the matrix  whose row and column
indices belong to $I$.  Note that $M_I$ is a $(r+s)\times (r+s)$ $q$-super Manin matrix.

Let $\Kc_{r|s}$ be the subalgebra of $\Kc_{m|n}$ with generators $\psi_i, x_j$ for $i,j \in I$.
Considering the right coacting of Hopf superalgebra $\widetilde{\Mc}_{r|s}$ on $\Kc_{r|s}$ by
\beq
\bal
\delta_{\Kc_{r|s}}: \Kc_{r|s} &\rightarrow\Kc_{r|s} \ot \widetilde{\Mc}_{r|s}\\
\delta_{\Kc_{r|s}}(\Psi_{I_1}) &= \Psi_{I_1} \ot  M_{I},\\
\delta_{\Kc_{r|s}}(X^t_{I_2})&= X^t_{I_2} \ot \imath(M_{I}^{st}).
\eal
\eeq
Define the minor quantum Berezinian by
\beq
\delta_{H_{r|s}}(\mathop{\prod}\limits^{\rightarrow}_{i\in I_1}\psi_{i} \mathop{\prod}\limits^{\leftarrow}_{j\in I_2}x_{j}) =\mathop{\prod}\limits^{\rightarrow}_{i \in I_1}\psi_{i} \mathop{\prod}\limits^{\rightarrow}_{j\in I_2}x_{j} \ot \Ber_q(M_{I}).
\eeq

\section{Minor identities for Berezinians}
For an $N\times N$ matrix $A=(a_{ij})$ with possible non-commutative entries, denote by $A^{ij}$ the submatrix obtained from $A$ by deleting the $i$-th row and $j$-th column. Suppose that the matrix $A^{ij}$ is invertible. The quasideterminants $|A|_{ij}$ is defined by
$$|A|_{ij}=a_{ij}-r_i^j(A^{ij})^{-1}c_j^i,$$
where $r_i^j$ is the row matrix obtained from the $i$-th row of $A$ by deleting the element $a_{ij}$, $c_j^i$ is the column matrix obtained from the $j$-th column of $A$ by deleting the element $a_{ij}$. Denote by
\beq
|A|_{ij}=
\left|\begin{array}{ccccc}
    a_{11}&\cdots& a_{1j}&\cdots&a_{1N}\\
    &\cdots&  & \cdots& \\
    a_{i1}&\cdots&\framebox{$a_{ij}$}& \cdots &a_{iN}\\
    &\cdots&  &\cdots& \\
    a_{N1}&\cdots&a_{Nj}&\cdots & a_{NN}
\end{array}\right|.
\eeq
If there exists the inverse matrix $A^{-1}$ and its $ji$-th entry $(A^{-1})_{ji}$ is an invertible element. Then $|A|_{ij}=((A^{-1})_{ji})^{-1}$ (see \cite{GR2}). We denote by $A^{(k)}$ (resp. $A_{(k)}$) the principal submatrix of
$A$ of the form $(A_{ij})_{1\leq i,j\leq k}$ (resp. $(A_{ij})_{m+n-k+1\leq i,j\leq m+n}$ ).

\begin{thm}\label{quasi-decomp}
We have the following decomposition of
 $\Ber_q(M)$ in  $\widetilde{\Mc}_{m|n}$,
\beq
\bal
	\Ber_q(M)=\big|M^{(1)}\big|_{11}\cdots
	\big|M^{(m)}\big|_{mm} \big|M^{(m+1)}\big|^{-1}_{m+1,m+1}\cdots
	\big|M^{(m+n)}\big|^{-1}_{m+n,m+n}.
\eal
\eeq
\end{thm}
\begin{proof}
  Note that in the  right quantum algebra $\Mc_m$
  \beq\label{quasi-decom1}
  \sum_{\si \in \Sym_m}(-q)^{-l(\si)}M_{\si(1),1}M_{\si(2),2}\cdots M_{\si(m),m}=\big|M^{(1)}\big|_{11}\cdots
	\big|M^{(m)}\big|_{mm}.
  \eeq
The left quantum superalgebra $\Mc_{n|m}^{\circ}$ is defined by
\beq
\mathcal{H}_{2}^{q}M^{\circ}_1M^{\circ}_2\mathcal{A}_{2}^{q}=0.
\eeq
Then there is a superalgebra morphism $\varrho_{m|n}$ from the right quantum superalgebra $\Mc_{m|n}^{q}$ to left quanntum superalgebra $\Mc_{n|m}^{\circ}$:
 \beq
 \bal
 \varrho_{m|n}:\Mc_{m|n}^{q}&\rightarrow \Mc_{n|m}^{\circ}\\
M_{ij}&\mapsto  M_{m+n-i+1,m+n-j+1}^{\circ}.
 \eal
 \eeq

It  follows that  $\varrho_{m|n}$ is an isomorphism by the defining relations of the right and left quantum superalgebras.
Since $(M^{-1})_{ij}=\big|M\big|_{ji}^{-1}$,
the image of the second factor on the right side of \eqref{Bere}  under the isomorphism $\varrho_{m|n}$ is the determinant
\beq
\sum_{\rho \in \Sym_n}(-q)^{-l(\rho)}(M^{\circ})^{-1}_{n,\rho(n)}\cdots (M^{\circ})^{-1}_{1,\rho(1)}.
\eeq
The corresponding quasideterminant decomposition is proved in the same way as
\eqref{quasi-decom1}, so that
\beq
\sum_{\rho \in \Sym_n}(-q)^{-l(\rho)}(M^{\circ})^{-1}_{n,\rho(n)}\cdots (M^{\circ})^{-1}_{1,\rho(1)}=\big|((M^{\circ})^{-1})^{(n)}\big|_{nn}\cdots \big|((M^{\circ})^{-1})^{(1)}\big|_{11}.
\eeq
Since the preimage of  $\big|((M^{\circ})^{-1})^{(k)}\big|_{kk}$ under $\varrho_{m|n}$ is  $\big|M^{(m+n-k+1)}\big|^{-1}_{m+n-k+1,m+n-k+1}$, the preimage of the determinant is
\beq
\big|M^{(m+1)}\big|^{-1}_{m+1,m+1}\cdots
	\big|M^{(m+n)}\big|^{-1}_{m+n,m+n}.
\eeq
\end{proof}

\begin{cor}\label{pi-tran-quasi-decom}
We have the following decomposition of
 $\Ber_q((M^{\Pi})^{st})$ in  $\widetilde{\Mc}_{m|n}$,
\beq
\bal
	\Ber_q(M^{\Pi\circ st})=\big|M_{(n)}\big|_{m+1,m+1}\cdots
   \big|M_{(1)}\big|_{m+n,m+n}
   \big|M_{(m+n)}\big|^{-1}_{11}\cdots
	\big|M_{(n+1)}\big|^{-1}_{mm}.
\eal
\eeq
\end{cor}
\begin{proof}
  By relation  \eqref{pi-tran-rela} of $\Pi$-transpose and the inverse of $M$, we have that
\beq
\bal \Ber_q(M^{\Pi\circ st})=\Ber_{q^{-1}}(M^{-1})=& \sum_{\rho \in \Sym_n}(-q)^{l(\rho)}M_{m+1,m+\rho(1)}\cdots M_{m+n,m+\rho(n)}\\
	&\times \sum_{\si \in \Sym_m}(-q)^{l(\si)}(M^{-1})_{\si(1),1}(M^{-1})_{\si(2),2}\cdots (M^{-1})_{\si(m),m}\\
	=&\big|(M^{-1})^{(m+1)}\big|^{-1}_{m+1,m+1} \cdots
	\big|(M^{-1})^{(m+n)}\big|^{-1}_{m+n,m+n}\\
	&\times \big|(M^{-1})^{(1)} \big|_{11} \cdots
	\big|(M^{-1})^{(m)}\big|_{mm}.
\eal
\eeq
It follows from Jacobi's ratio theorem \cite{GR1,KL}, that
\begin{equation}
\begin{aligned}
   &\big|(M^{-1})^{(m+1)}\big|^{-1}_{m+1,m+1} \cdots \big|(M^{-1})^{(m+n)}\big|^{-1}_{m+n,m+n}
   \big|(M^{-1})^{(1)} \big|_{11} \cdots
	\big|(M^{-1})^{(m)}\big|_{mm} \\
   &=\big|M_{(n)}\big|_{m+1,m+1}\cdots
   \big|M_{(1)}\big|_{m+n,m+n}
   \big|M_{(m+n)}\big|^{-1}_{11}\cdots
	\big|M_{(n+1)}\big|^{-1}_{mm}.
\end{aligned}
\end{equation}
\end{proof}

For any $i_1,\ldots,i_k,j_1,\ldots,j_k\in [m+n]$,
we denote by $A^{i_1,\ldots,i_k}_{j_1,\ldots,j_k}$ the matrix whose
$ab$-th entry is $A_{i_a j_b}$.
We have an analog of Proposition \ref{permutation} for quasideterminant.
\begin{prop}\label{permutation-qiasi}
\beq
\bal \Ber_q(M)=&\frac{\Inv(J_1)\Inv(J_2)}{\Inv(I_1)\Inv(I_2)}M_{i_1,j_1}\big|M^{i_1,i_2}_{j_1,j_2}\big|_{i_2,j_2}\cdots
	\big|M^{i_1,\ldots,i_m}_{j_1,\ldots,j_m}\big|_{i_m,j_m}\\
    & \times \big|M^{i_1,\ldots,i_{m+1}}_{j_1,\ldots,j_{m+1}}\big|^{-1}_{i_{m+1},j_{m+1}}\cdots
	\big|M^{i_1,\ldots,i_{m+n}}_{j_1,\ldots,j_{m+n}}\big|^{-1}_{i_{m+n},j_{m+n}}.
\eal
\eeq
where  $I_1=(i_1,\ldots,i_m), J_1=(j_1,\ldots,j_m)$ (resp. $I_2=(i_{m+1},\ldots,i_{m+n}), J_2=(j_{m+1},\ldots,j_{m+n})$) be two any rearrangements of $[m]$ (resp. $[m+n]\setminus[m]$).
\end{prop}
\begin{proof}
By the quasideterminant decompositions for quantum determinant in \cite[Theorem 3.3]{GR1} and  \cite[Theorem 3.1]{KL}, we have that
\beq
\bal
\Ber_q(M)
&=\big|M^{(1)}\big|_{11}\cdots
   \big|M^{(m)}\big|_{mm}\big|M^{(m+1)}\big|^{-1}_{m+1,m+1}\cdots
   \big|M^{(m+n)}\big|^{-1}_{m+n,m+n}\\
&=\frac{\Inv(J_1)}{\Inv(I_1)}M_{i_1,j_1}\big|M^{i_1,i_2}_{j_1,j_2}\big|_{i_2,j_2}\cdots
  \big|M^{i_1,\ldots,i_m}_{j_1,\ldots,j_m}\big|_{i_m,j_m}\big|M^{(m+1)}\big|^{-1}_{m+1,m+1}\cdots
  \big|M^{(m+n)}\big|^{-1}_{m+n,m+n},
\eal
\eeq
considering the image of the second factor under  the isomorphism  $\varrho_{m|n}$ which is the quantum determinant of left quantum superalgebra. So the second factor can be decomposed in the same way.
\end{proof}
We have the following theorem that is an analog of Jacobi's ratio theorem for the quantum Berezinian.
\begin{thm}\label{Jac rat}
Let $I=\{i_1<\cdots<i_k\}$ be a subset of $[m+n]$  and
$I^{c}=\{i_{k+1}<\cdots<i_{m+n}\}$ be its complement in $[m+n]$.
If subset $I$ satisfies one of the conditions: (i) $I\subseteq [m]$; (ii) $[m]\subseteq I$, then
\beq
\Ber_q(M)= \Ber_q(M_I) \Ber_{q^{-1}}(((M^{-1})_{I^c})^{\Pi \circ st}).
\eeq
\end{thm}
\begin{proof}
We first prove the theorem in the case $I \subseteq [m]$.
According to Theorem \ref{quasi-decomp} and proposition \ref{permutation-qiasi},
\beq
\begin{split}
\Ber_q(M)
=&\big|M^{i_1}_{i_1}\big|_{i_1i_1}
\big|M^{i_1 i_2}_{i_1i_2}\big|_{i_2i_2}
\cdots\big|M^{i_1,\ldots,i_m}_{i_1,\ldots,i_m}\big|_{i_mi_m}\big|M^{(m+1)}\big|^{-1}_{m+1,m+1}\cdots
	\big|M^{(m+n)}\big|^{-1}_{m+n,m+n}\\
=&\Ber_q(M_I) \big|M^{i_1,\ldots,i_{k+1}}_{i_1,\ldots,i_{k+1}}\big|_{i_{k+1},i_{k+1}}
\cdots\big|M^{i_1,\ldots,i_m}_{i_1,\ldots,i_m}\big|_{i_mi_m}\big|M^{(m+1)}\big|^{-1}_{m+1,m+1}\cdots
	\big|M^{(m+n)}\big|^{-1}_{m+n,m+n}.
\end{split}
\eeq
By Jacobi's ratio theorem \cite{GR1,KL} for quasideterminant  and corollary \ref{pi-tran-quasi-decom},
\beq
\begin{split}
&\big|M^{i_1,\ldots,i_{k+1}}_{i_1,\ldots,i_{k+1}}\big|_{i_{k+1},i_{k+1}}
\cdots \big|M^{i_1,\ldots,i_m}_{i_1,\ldots,i_m}\big|_{i_m,i_m}\big|M^{(m+1)}\big|^{-1}_{m+1,m+1}\cdots
\big|M^{(m+n)}\big|^{-1}_{m+n,m+n}\\
=&\big|(M^{-1})^{i_{k+1},\ldots,i_{m+n}}_{i_{k+1},\ldots,i_{m+n}}\big|_{i_{k+1},i_{k+1}}^{-1}\cdots
\big|(M^{-1})^{i_m,\ldots,i_{m+n}}_{i_m,\ldots,i_{m+n}}\big|_{i_{m},i_{m}}^{-1}
\\
&\times\big|(M^{-1})^{m+1,\ldots,m+n}_{m+1,\ldots,m+n}\big|_{m+1,m+1}\ldots\big|(M^{-1})^{m+n}_{{m+n}}\big|_{m+n,m+n}\\
=&\Ber_{q^{-1}}(((M^{-1})_{I^c})^{\Pi \circ st}).
\end{split}
\eeq

 When $[m]\subseteq I$, by Theorem \ref{quasi-decomp} and proposition \ref{permutation-qiasi}, we have that
\beq
\begin{split}
\Ber_q(M)
=&\big|M^{(1)}\big|_{11}\cdots \big|M^{(m)}\big|_{mm}\big|M^{1,\ldots,m,i_{k-m}}_{1,\ldots,m,i_{k-m}}\big|^{-1}_{i_{k-m}i_{k-m}}
\cdots \\
&\cdots \big|M^{1,\ldots,m,i_{k-m},\dots,i_k}_{1,\ldots,m,i_{k-m},\ldots,i_k}\big|^{-1}_{i_{k}i_{k}}
\cdots\big|M^{1,\ldots,m,i_{k-m},\ldots,i_{m+n}}_{1,\ldots,m,i_{k-m},\ldots,i_{m+n}}\big|^{-1}_{i_{m+n}i_{m+n}}\\
=&\Ber_q(M_I) \big|M^{1,\ldots,m,i_{k-m},\dots,i_{k+1}}_{1,\ldots,m,i_{k-m},\ldots,i_{k+1}}\big|^{-1}_{i_{k+1}i_{k+1}}
\cdots\big|M^{1,\ldots,m,i_{k-m},\ldots,i_{m+n}}_{1,\ldots,m,i_{k-m},\ldots,i_{m+n}}\big|^{-1}_{i_{m+n}i_{m+n}}.
\end{split}
\eeq
By Jacobi's ratio theorem for quasideterminant  and corollary \ref{pi-tran-quasi-decom},
\beq
\begin{split}
&\big|M^{1,\ldots,m,i_{k-m},\dots,i_{k+1}}_{1,\ldots,m,i_{k-m},\ldots,i_{k+1}}\big|^{-1}_{i_{k+1}i_{k+1}}
\cdots\big|M^{1,\ldots,m,i_{k-m},\ldots,i_{m+n}}_{1,\ldots,m,i_{k-m},\ldots,i_{m+n}}\big|^{-1}_{i_{m+n}i_{m+n}}\\
=&\big|(M^{-1})^{i_{k+1},\ldots,i_{m+n}}_{i_{k+1},\ldots,i_{m+n}}\big|_{i_{k+1},i_{k+1}}\cdots
\big|(M^{-1})^{i_{m+n}}_{i_{m+n}}\big|_{i_{m+n},i_{m+n}}\\
=&\Ber_{q^{-1}}(((M^{-1})_{I^c})^{\Pi \circ st}).
\end{split}
\eeq
\end{proof}

Taking $I=\emptyset$ in the above theorem, we immediately get the similar relation to \eqref{pi-tran-rela}:
\begin{cor}\label{ber inverse}
\beq
\Ber_q(M)=\Ber_{q^{-1}}((M^{-1})^{\Pi  \circ st}).
\eeq
\end{cor}

The following is an analogue of Schur's complement theorem.
\begin{thm}\label{Sch com}
Write $M$
in block matrix,
\[
M=\left(\begin{array}{cc}
	M_{11}&M_{12}\\
	M_{21}&M_{22}\\
\end{array}\right)
\]
such that $M_{11}$ and $M_{22}$ are submatrices of size $k\times k$ and $(m+n-k)\times(m+n-k)$ respectively.
Then
\beq
\Ber_q(M)=\Ber_q(M_{11}) \Ber_{q} \left(M_{22}-M_{21}M_{11}^{-1}M_{12} \right).
\eeq
\end{thm}
\begin{proof}
Let
\[
N=M^{-1}=\left(\begin{array}{cc}
	N_{11}&N_{12}\\
	N_{21}&N_{22}\\
\end{array}\right).
\]
It is well known that
\[
N_{22}=(M_{22}-M_{21}M_{11}^{-1}M_{12})^{-1}.
\]
By Theorem \ref{Jac rat},
\beq
\Ber_q(M)=\Ber_q(M_{11})  \Ber_{q^{-1}}((N_{22})^{\Pi \circ st}).
\eeq
It follows from Corollary \ref{ber inverse} that
\beq
  \Ber_{q^{-1}}((N_{22})^{\Pi \circ st}) =\Ber_{q} \\ \left(M_{22}-M_{21}M_{11}^{-1}M_{12} \right).
\eeq
This completes the proof.
\end{proof}

Taking $k=m$ in theorem \ref{Sch com} we have the following formula which is analog of the definition of Berezinian for the classical super matrix.
\begin{cor}
If $k=m$ in Theorem \ref{Sch com}, then
\beq
\Ber_q(M)={\det}_q(M_{11}){\det}_{q^{-1}} \left(M_{22}-M_{21}M_{11}^{-1}M_{12} \right)^{-1}.
\eeq
\end{cor}

\begin{thm}[Cayley's complementary identity]\label{Cay com}
Suppose a minor identity for quantum Berezinians is given:
\beq
\sum_{i=1}^{k}b_i\prod_{j=1}^{m_i}\Ber_q(M_{I_{ij}})=0,
\eeq
where $I_{ij}$ are increasing  multi-indices satisfying one of the conditions: (i) $I_{ij} \subseteq [m+n]\setminus [m]$; (ii) $[m+n]\setminus [m]\subseteq I_{ij}$, and $b_i\in \CC(q)$, then when $I_{ij}$ satisfies condition (i) or (ii), the following identity holds
\beq\label{cayley-com2}
\sum_{i=1}^{k}b'_i\prod_{j=1}^{m_i}\Ber_q(M_{I_{ij}^c})^{-1}\Ber_q(M)=0,
\eeq
where $b'_i$ is obtained from $b_i$ by replacing $q$ by $\qin$.
\end{thm}
\begin{proof}
By applying the minor identity to $(M^{-1})^{\Pi \circ st}$, we get that
\[
\sum_{i=1}^{k}b_i'\prod_{j=1}^{m_i}\Ber_{q^{-1}}\left((M^{-1}_{I_{ij}})^{\Pi\circ st}\right)=0,
\]
By theorem \ref{Jac rat} we obtain the equation \eqref{cayley-com2}.
\end{proof}

\begin{thm}[Muir's law]\label{Muir law}
Let $I=(i_1<\ldots<i_l) \subseteq [m+n]\setminus [m]$, Assume a minor identity for quantum Berezinians
\beq
\sum_{i=1}^{k}b_i\prod_{j=1}^{m_i}\Ber_q\left(M_{I_{ij}}\right)=0,
\eeq
where  $I_{ij} \subseteq  I$  are increasing  multi-indices and $b_i\in \CC(q)$.
  then the following identity holds
\beq
\sum_{i=1}^{k}b_i\prod_{j=1}^{m_i}\Ber_q\left(M_{I_{ij}\bigcup I^c}\right)\Ber_q\left(M_{{I^c}}\right)^{-1}
=0.
\eeq
\end{thm}
\begin{proof}
Note that $I, I \setminus I_{ij}\subseteq [m+n]\setminus [m]$, it results from two successive applications of Theorem \ref{Cay com}. The first application in $M_I$ shows that
\beq
\sum_{i=1}^{k}b'_i\prod_{j=1}^{m_i}\Ber_q(M_{I \setminus I_{ij}})^{-1}\Ber_q(M_I)=0.
\eeq
And by the second application in $M$, we obtain that
\beq
\sum_{i=1}^{k}b_i\prod_{j=1}^{m_i}\Ber_q\left(M_{I_{ij}\bigcup I^c}\right)\Ber_q\left(M_{{I^c}}\right)^{-1}
=0.
\eeq
\end{proof}

For any $k\geq 0$, we define the homomorphism
\beq
\bal
\phi_k: \mathcal{M}_{m|n}&\rightarrow \mathcal{M}_{k+m|n}\\
M_{ij}&\mapsto M_{k+i,k+j},
\eal
\eeq
and an anti-isomorphism of $\widetilde{\Mc}_{m|n}$:
\beq
\om_{m|n}(M)=M^{-1}.
\eeq
Considering the composition
\beq
\psi_k=\om_{k+m|n}\circ\phi_k\circ\om_{m|n}.
\eeq
By the same argument in \cite[Lemma 1.11.2 ]{Mo1}, we have the lemma.
\begin{lem}
For any $1\leq i,j\leq m+n$, we have
\beq
\psi_k(M_{ij})=\left|\begin{array}{cccc}
	M_{11}&\cdots&M_{1k}&M_{1,k+j}\\
	\vdots&&\vdots&\vdots\\
	M_{k1}&\cdots&M_{kk}&M_{k,k+j}\\ M_{k+i,1}&\cdots&M_{k+i,k}&\framebox{$M_{k+i,k+j}$}
\end{array}\right|.
\eeq
\end{lem}
The following Sylvester theorem for the quantum Berezinian follows from the above Lemma.
\begin{thm}\label{syl}
Denote the generator matrix of $\mathcal{M}_{k+m|n}$ by $M_{k+m|n}$, then
\beq
\psi_k(\Ber_q(M))={\det}_{q}(M^{1,\ldots, k}_{1,\ldots, k})^{-1}\Ber_q(M_{k+m|n}).
\eeq
\end{thm}
\begin{proof}
Write $M_{k+m|n}$ in block matrix,
\[
M_{k+m|n}=\left(\begin{array}{cc}
	M_{11}&M_{12}\\
	M_{21}&M_{22}\\
\end{array}\right)
\]
such that $M_{11}=M^{1,\ldots, k}_{1,\ldots, k}$, $M_{22}=M$ and $M_{12}$ and $M_{21}$ are submatrices of size $k\times (m+n)$ and $(m+n)\times k$ respectively.
By Schur's complement theorem \ref{Sch com}, we have that
\beq
{\det}_{q}(M^{1,\ldots, k}_{1,\ldots, k})^{-1} \Ber_q(M_{k+m|n})=\Ber_{q} \left( M-M_{21}M_{11}^{-1}M_{12} \right).
\eeq
Moreover, from the definition of quasideterminant,
\[
\psi_k(M_{ij})=(M-M_{21}M_{11}^{-1}M_{12})_{k+i,k+j}.
\]
Hence,
\beq
\psi_k(\Ber_q(M))=\Ber_{q} \left( M-M_{21}M_{11}^{-1}M_{12} \right).
\eeq
This completes the proof.
\end{proof}

%
%
%
%

\bigskip
\centerline{\bf Acknowledgments}
\medskip
The work is supported in part by the National Natural Science Foundation of China (grant nos.
12171303 and 112571026), the Simons Foundation (grant no. MP-TSM-0002518),
and the Fundamental Research Funds for the Central Universities (grant no. CCNU25JCPT031).

\bigskip

\bibliographystyle{amsalpha}

\small

\noindent
Naihuan Jing\\
Department of Mathematics\\
North Carolina State University, Raleigh, NC 27695, USA\\
jing@ncsu.edu

\vspace{5 mm}

\noindent
Yinlong Liu\\
Department of Mathematics\\
Shanghai University, Shanghai 200444, China\\
yinlongliu@shu.edu.cn

\vspace{5 mm}

\noindent
Jian Zhang\\
School of Mathematics and Statistics\\
Central China Normal University, Wuhan, Hubei 430079, China\\
jzhang@ccnu.edu.cn

\end{document}